\documentclass[12pt]{amsart}
\usepackage[latin1]{inputenc}
\usepackage{geometry}
\usepackage{xcolor}
\usepackage{float}
\usepackage{verbatim}
\usepackage{graphics}

\usepackage{graphicx}
\usepackage{amsfonts}
\usepackage{amsthm}
\usepackage{amssymb}
\usepackage[centertags]{amsmath}
\usepackage{amssymb}
\usepackage{color}
\usepackage{tikz}
\usepackage{enumerate}
\usepackage{hyperref}
\usepackage{subfigure}

\theoremstyle{plain}
\newtheorem{theorem}{Theorem}[section]
\newtheorem{lemma}[theorem]{Lemma}

\newtheorem{definition}[theorem]{Definition}

\makeatother

\title[Recursion formula for the volumes of moduli spaces]{Recursion formula for the volumes of moduli spaces of compact hyperbolic surfaces with cone points}

    \author{Haoyang Jiang and Lixin Liu}
    \address{School of Mathematics, Sun Yat-Sen University, 510275, Guangzhou, PR China. }
    \email{jianghaoyang@zsdx.wecom.work}
    \address{School of Mathematics, Sun Yat-Sen University, 510275, Guangzhou, PR China. }
    \email{mcsllx@mail.sysu.edu.cn}

\date{\today}

\usepackage{microtype}

\begin{document}
\maketitle

\begin{abstract}Let $V_{g,m,n}(\overrightarrow L,\overrightarrow \theta)$ be the Weil-Petersson volume of the moduli space of hyperbolic surfaces of genus g with m geodesic boundary components of length $\overrightarrow L=(\ell_1,...,\ell_m)$ and $n$ cone points of angle $\overrightarrow \theta=(\theta_1,...\theta_n)$. By using the generalized McShane's identities, we show that $V_{g,m,n}(\overrightarrow L,\overrightarrow \theta)$ is a polynomial of $(\ell_1,...,\ell_n,i\theta_1,...,i\theta_m)$.  And we obtain a recursion formula for $V_{g,m,n}(\overrightarrow L,\overrightarrow \theta)$, which is a generalization of Mirzakhani's result.
    \par\textbf{keywords}: McShane's identity, Volume polynomials, Recursion formula, Mirzakhani's theory
\end{abstract}
\setcounter{page}{1}
\maketitle
\thispagestyle{plain}

\section{Introduction}\label{sec1}
Denote a compact hyperbolic surface of genus $g$ with $m$ geodesic boundary components and $n$ cone points by $S_{g,m,n}$. Denote the moduli space of $S_{g,m,n}$ by $\mathcal M_{g,m,n}$, the mapping class group of $S_{g,m,n}$ by $Mod_{g,m,n}$ and the Teichm\"uller space of $S_{g,m,n}$  by $\mathcal T_{g,m,n}$. Given the lengths of the $m$ geodesic boundaries $\overrightarrow L=(\ell_1,\cdots,\ell_m)$ and the angles of the $n$ cone points $\overrightarrow \theta=(\theta_1,\cdots,\theta_n)$, we denote the surface by $S_{g,m,n}(\overrightarrow L,\overrightarrow \theta)$, the moduli space of the surface by $\mathcal M_{g,m,n}(\overrightarrow L,\overrightarrow \theta)$, the volume of the moduli space by  $V_{g,m,n}(\overrightarrow L,\overrightarrow \theta)$.

Mirzakhani\cite{MR2264808} established a method for integrating over the moduli space based on the McShane's identities for hyperbolic surface with boundaries. Using this method, she proved that the volume of the moduli space of hyperbolic surface with boundaries is a polynomial of the length of the geodesic boundaries of the hyperbolic surface. Furthermore she obtained a recursion formula for the  volumes of the moduli spaces of hyperbolic surfaces with boundaries. It's a folklore that Mirzakhani's results are still true for compact hyperbolic surface with cone points for all angles of cones in $(0,\pi]$. We'll give a proof for this folklore.

In this paper, we prove the following results:
    \begin{theorem}\label{thm 1.1}
    Given a compact hyperbolic surface $S_{g,m,n}(\overrightarrow L,\overrightarrow \theta)$ with all cone points with angles in $(0,\pi]$, there exists a recursion formula for the volume  $V_{g,m,n}(\overrightarrow L,\overrightarrow \theta)$ according to the splitting of $S_{g,m,n}(\overrightarrow L,\overrightarrow \theta)$.
\end{theorem}
The details of the recursion formula are shown in Section \ref{sec4}.
As we know, there exists pants decomposition for compact hyperbolic surface with cone points for all angles of cones in $(0,\pi]$ \cite{Zhang2004}. But for the compact hyperbolic surface with cone points which has an angle $\pi<\theta<2\pi$, the pants decomposition may not exist, which was shown in the appendix of \cite{anagnostou2024volumesmodulispaceshyperbolic}. Because Mirzakhani's method depends on the pants decomposition of hyperbolic surface, all the compact hyperbolic surfaces with cone points in this paper have cone points with all angles in $(0,\pi]$. For more works related to volumes of moduli spaces of hyperbolic surfaces with cone points, we refer to \cite{anagnostou2023weilpeterssonvolumesstabilityconditions}\  \cite{anagnostou2024volumesmodulispaceshyperbolic} and \cite{MR2520065}. Do and Norbury \cite{MR2520065} obtained  the recursion formula when $\theta \longrightarrow 2\pi$.   Anagnostou and Norbury \cite{anagnostou2024volumesmodulispaceshyperbolic} developed a method of extending the homeomorphism between the Deligne-Mumford compactification of $\mathcal M_{g,n}(\overrightarrow \theta)$ and the Deligne-Mumford compactification of $\mathcal M_{g,n}$. Using this method they extended the volume polynomials of $\mathcal M_{g,n}(\overrightarrow L)$ to that of  $\mathcal M_{g,n}(\overrightarrow \theta)$(Theorem 1.2 in \cite{anagnostou2024volumesmodulispaceshyperbolic}).
    This result focus on the surfaces $\mathcal S_{g,n}(\theta)$ with cone points and without geodesic boundary components. For our result, the surfaces have both cone points and geodesic boundary and the recursion formula is obtained. For the moduli space of decorated surface, Goncharov and Zhe Sun defined the exponential volume which is a generalization of the volume in Weil-Petersson volume form, and they also get a recursion formula.\cite{goncharov2024exponentialvolumesmodulispaces}

\section{Results}\label{sec2}

\subsection{McShane's identity}\label{subsec1}

McShane proved the following identity :

\begin{theorem}[McShane~\cite{MR3389436}]
In a one-punctured hyperbolic torus $T$,
\begin{equation}\notag
\sum_{\gamma}\frac{1}{1+\exp|\gamma|}=\frac{1}{2},
\end{equation}
where the sum is over all simple closed geodesics on $T$ and $|\gamma|$ denotes the hyperbolic length of $\gamma$ on $T$.

\end{theorem}

Then McShane generalized the above identity to more general hyperbolic surfaces:
\begin{theorem}
    [McShane~\cite{MR1625712}]
In a finite area hyperbolic surface $M$ with cusps and without boundary,
\begin{equation}\notag
\sum\frac{1}{1+\exp\frac{1}{2}(|\alpha|+|\beta|)}=\frac{1}{2},
\end{equation}
where the sum is over all unordered pairs of simple closed geodesics $\{\alpha,\ \beta\}$(where $\alpha$ or $\beta$ might be a cusp treated as a simple closed geodesic of length 0) on $M$ such that $\alpha,\ \beta$ bound with a distinguished cusp point an embedded pair of pants on $M$.
\end{theorem}

There are various generalizations of the McShane's identity. For more details we refer to the survey paper
 \cite{MR3497292}.

\begin{definition}
    Let $S_{g,n,m}$ be an oriented compact surface of genus $g$ with $n$
boundary components and $m$ marked points such that $2g-n-m >0$.
By definition, a hyperbolic surface is a hyperbolic metric on $S_{g,n,m}$
with $n$ geodesic boundary components and $m$ conical singularities at the marked points.
\end{definition}

Tan, Wang and Zhang generalized the McShane's identity to compact hyperbolic surface with cone points. They obtained the following result in Theorem 1.5 of  \cite{MR2215456}:
\begin{theorem}(Generalized McShane's Identity)\label{thm 2.4}

Let $M$ be a compact hyperbolic surface, either with a single cone point of cone angle $\theta\in(0,\pi]$ or with a single boundary component of length $\ell>0$. Then we have respectively
\begin{equation}\notag
 \sum 2\tan^{-1}\left(\frac{\sin\frac{\theta}{2}}{\cos\frac{\theta}{2}+\exp(\frac{|\alpha|+|\beta|}{2})}\right)=\frac{\theta}{2} , \end{equation}
\begin{equation}\notag
    \sum 2\tanh^{-1}\left(\frac{\sinh\frac{\ell}{2}}{\cosh\frac{\ell}{2}+\exp(\frac{|\alpha|+|\beta|}{2})}\right)=\frac{\ell}{2},
    \end{equation}
where the sum in either case extends over all unordered pairs of simple closed geodesics on M which bound with the cone point (respectively, the boundary geodesics) an embedded pair of pants.
\end{theorem}
These results can be used to calculate the volumes of the moduli spaces of compact hyperbolic surfaces with cone points.

Besides, Yi Huang and Zhe Sun generalized the McShane's identity for higher Teichm\"uller theory.\cite{MR4595283}
\subsection{The Fenchel-Nielsen coordinates}\label{subsec2}

A surface $S$ admits a pair of pants decomposition  $\mathcal{P}=\{P_i\ \vert \ i=1,2,\cdots, q\}$ in which each $P_i$ is a generalized pair of {\it pants}, that is, a topological sphere with three holes, where a {\it hole} is either a point removed or an open disk removed.

Fixed a pant decomposition $\mathcal{P}=\{P_i\ \vert \ i=1,2,\cdots, q\}$ of the surface $S$. Let $\mathcal{L}$ be the set of boundary curves of $P_i,\ i=1,2, \cdots, q$. For a hyperbolic metric $H$ on $S$, each homotopy class of curves $C_i \in \mathcal{L}$ we associate a length parameter and a twist parameter for it. The length parameter $l_H(C_i)$ of $C_i$ for $H$ is the length of geodesic in the homotopic class of $C_i$ for the hyperbolic metric $H$. The precise definition of the twist parameter $\theta_H(C_i)$ of $C_i$ for the metric $H$ can be found in Theorem 4.6.23 of \cite{Thurston1979TheGA}, we omit the details. Then the Fenchel-Nielsen coordinate of the hyperbolic metric $H$ for the pant decomposition $\mathcal{P}=\{P_i\ \vert \ i=1,2,\ldots, q\}$ is $\{l_H(C_1), \theta_H(C_1), l_H(C_2), \theta_H(C_2), \cdots,  l_H(C_n), \theta_H(C_n), \cdots \}$ for $C_i \in \mathcal{L}$ .

 Let $S_{g,m,n}(\overrightarrow L,\overrightarrow \theta)$ be a compact hyperbolic surface with cone points, where $\overrightarrow L=(\ell_1,\cdots,\ell_m),\ \overrightarrow \theta=(\theta_1,\cdots,\theta_n)$. Let $\mathcal{T}_{g,m,n}(\overrightarrow L,\overrightarrow \theta)$ be the Teichm\"uller space of $S_{g,m,n}(\overrightarrow L,\overrightarrow \theta)$. We can use $6g-6+2n+2m$ real parameters which comes from  the Fenchel-Nielsen coordinates  to describe the Teichm\"uller space $\mathcal{T}_{g,m,n}(\overrightarrow L,\overrightarrow \theta)$, which consists of the set of the lengths of the geodesic boundaries and the set of the twisting parameters used to glue the pieces.

 We have an isomorphism \cite{buser1992geometry}
\begin{equation}\notag
    \mathcal{T}_{g,m,n}(\overrightarrow L,\overrightarrow \theta)\cong \mathbb{R}_{+}^{\mathcal{R}}\times \mathbb{R}^{\mathcal{P}}
\end{equation}

by the map

$$X \longrightarrow (\ell_{{\alpha}_{i}}(X),\tau_{{\alpha}_{i}}(x)).$$
By the work of Goldman \cite{MR762512} and Wolpert \cite{MR657237}, the Weil-Petersson symplectic structure has a simple form in Fenchel-Nielsen coordinates for Teichm\"uller space:
\begin{theorem}
    The Weil-Petersson symplectic form is given by
    \begin{equation}\notag
\omega_{wp}=\sum_{i=1}^{k}d\ell_{{\alpha}_{i}}\bigwedge d\tau_{{\alpha}_{i}}.
    \end{equation}

\end{theorem}
In order to generalize Mirzakhani's results to the compact hyperbolic surface with cone points, we need a generalization of the above theorem. This is based on the existence of the pants decomposition of the compact hyperbolic surface with cone points with all cone points $\in (0,\pi]$. The generalization of Theorem \ref{thm 2.5} was proven in page 3246 of \cite{MR1844999}:
\begin{theorem}\label{thm 2.6}
    For a compact hyperbolic surface with cone points with all cone points $\in (0,\pi]$, the Weil-Petersson symplectic form of its Teichm\"uller space is given by
    \begin{equation}\label{thm 2.5}
\notag\omega_{wp}=\sum_{i=1}^{k}d\ell_{{\alpha}_{i}}\bigwedge d\tau_{{\alpha}_{i}}.
    \end{equation}

\end{theorem}
And by Theorem 2.2 in \cite{MR1844999}, we know that $\omega_{wp}$ is invariant under the action of mapping class group. Thus it carries a volume form invariant under the action of the mapping class group, which is defined by
\begin{equation}\notag
    v_{wp}=\frac{1}{(3g-3+n)!}\omega_{wp}\wedge\cdots\wedge\omega_{wp}.
\end{equation}
\subsection{Twisting}\label{subsec3}
Given a simple closed geodesic $\alpha$ on $X\in \mathcal{T}_{g,m,n}(\overrightarrow L,\overrightarrow \theta)$ and $t\in \mathbb{R}$, we define the twisting deformation $tw_{\alpha}^{t}$ of $X$ as follows.

Cutting the surface $X$ along $\alpha$, turning the left hand side of $\alpha$ in the positive direction by distance $t$ and reglue back, we obtained a new surface. The new surface is denoted by $tw_{\alpha}^{t}(X)$. As $t$ varies continuously, the resulting continuous path in the Teichm\"uller space $X\in \mathcal{T}_{g,m,n}(\overrightarrow L,\overrightarrow \theta)$ is the Fenchel-Nielsen deformation of $X$ along $\alpha$. For $t=\ell_{\alpha}(X)$, we have:
\begin{equation}\notag
    tw_{\alpha}^{t}(X)=\phi_{\alpha}(X),
\end{equation}
where $\phi_{\alpha}$ is a right Dehn twist about $\alpha$.

Because we have $\ell_{\alpha}(tw_{\alpha}^{t}(X))=\ell_{\alpha}(X)$ for every $t\in[0,\ell_{\alpha}]$, the vector field $Y $generated by twisting around $\alpha$ is $symplectically\ dual$ to the one form $d\ell_{\alpha}$, which means $\omega_{wp}(Y,.)=d\ell_{\alpha}$. Note that $\omega_{wp}$ is invariant under the action of the mapping class group. Therefore, the level set $\mathcal{M}_{g,m,n}^{\gamma}(\overrightarrow L,\overrightarrow \theta,x)=\ell^{-1}_{\alpha}(x)$
carries a natural volume form $v_x.$
\subsection{Spliting}\label{subsect4}
A complex $\gamma=\sum_{i=1}^{k}c_{i}\gamma_{i}\ (c_{i}\in \mathbb{R}_{+})$ is called a multi-curve on $S_{g,m,n}$ if $\gamma_{i}$'s are disjoint, essential, non-peripheral simple closed curves which are not homotopic to each other.
Given $\gamma_{i}$'s with length-vector $\overrightarrow L=(\ell_1,\cdots,\ell_m)$, we can study the surface $S_{g,m,n}(\gamma)$ obtained by cutting $S_{g,m,n}$ along $\gamma_{1},\cdots,\gamma_{k}$. Then $S_{g,m,n}(\gamma)$ is a (possibly disconnected) surface with $m+2k$ boundary components and $s(\gamma)$ connected components. Every connected component gives rise to two boundary components, $\gamma_{i}^{1}$ and $\gamma_{i}^{2}$ on $S_{g,m,n}(\gamma)$. And we have:
\begin{equation}\notag
    \partial(S_{g,m,n}(\gamma))=\{\beta_{1},\cdots,{\beta}_{m},\ \theta_{1},\cdots,\theta_{n}\}\bigcup\{{\gamma_{1}^1,\gamma_{1}^2\cdots,\gamma_{k}^1,\gamma_{k}^2}\},
\end{equation}

where $\beta$'s are the boundary components of $S_{g,m,n}$.
Given $\Gamma=(\gamma_{1},\cdots,\gamma_{k})$, $\overrightarrow L=(\ell_{1},\cdots,\ell_{k})$, $\overrightarrow \beta=(\beta_{1},\cdots,{\beta}_{m})$ and $\overrightarrow{x}=(x_1,\cdots,x_k)\in \mathbb{R}_{+}^{k}$, we consider the moduli space
\begin{equation}\notag
    \mathcal{M}(S_{g,m,n}(\gamma),\ \ell_{\Gamma}=\overrightarrow{x},\ \ell_{\overrightarrow\beta}=\overrightarrow L)
\end{equation}
of hyperbolic Riemann surfaces homeomorphic to $S_{g,m,n}(\gamma)$.

We define $V_{g,m,n}(\Gamma,\overrightarrow{x},\overrightarrow \beta,\overrightarrow L)$ as following,
\begin{equation}\notag
    V_{g,m,n}(\Gamma,\ \overrightarrow{x},\ \overrightarrow \beta,\ \overrightarrow L)=\operatorname{Vol}(\mathcal{M}(S_{g,m,n}(\gamma),\ \ell_{\Gamma}=\overrightarrow{x},\ \ell_{\overrightarrow\beta}=\overrightarrow L)),
\end{equation}
which is the volume of $\mathcal{M}(S_{g,m,n}(\gamma),\ \ell_{\Gamma}=\overrightarrow{x},\ \ell_\beta=\overrightarrow L).$
If $S_{g,m,n}(\gamma)$ is divided into $s$ connected components, we can write it as a union of these simpler surfaces:
\begin{equation}\label{2.0}
    S_{g,m,n}(\gamma)=\bigcup_{i=1}^{s}S_{g_i,m_i,n_i}, \ A_{i}=\partial S_{g_i,m_i,n_i}\in\partial(S_{g,m,n}(\gamma)),
\end{equation}
where $\sum_i^sm_i=m+2k,\ \sum_i^sn_i=n.$ And we have:
\begin{equation}\label{2.1}
    \mathcal{M}(S_{g,m,n}(\gamma),\ \ell_{\Gamma}=\overrightarrow{x},\ \ell_{\overrightarrow\beta}=\overrightarrow L)\cong \prod_{i=1}^{s}\mathcal{M}_{g_i,m_i,n_i}(\ell_{A_i}).
\end{equation}
Then we have:
\begin{equation}\label{2.2}
    V_{g,m,n}(\Gamma,\ \overrightarrow{x},\ \overrightarrow \beta,\ \overrightarrow L)= \prod_{i=1}^{s}V_{g_i,m_i,n_i}(\ell_{A_i}).
\end{equation}

\section{Integration over the moduli space}\label{sec3}

\subsection{Integration formula}\label{subsec 3.1}
Mirzakhani \cite{MR2264808} obtained the formula of the integration over the moduli space of the hyperbolic surface with geodesic boundaries. We'll generalize her results to a compact hyperbolic surface with cone points.
For a given set $A$ of homotopy classes of simple closed curves on $S_{g,m,n},$ $\operatorname{Stab}(A)$ is defined by:
\begin{equation}\notag
    \operatorname{A}=\{h\in Mod_{g,m,n}|h\cdot A=A\}\subset Mod_{g,m,n}.
\end{equation}
For a multi-curve $\gamma=\sum_{i=1}^{k}c_{i}{\gamma}_{i}$ on $S_{g,m,n}$, we define the symmetry group of $\gamma$, $Sym(\gamma)$, by
\begin{equation}\notag
Sym(\gamma)/\cap_{i=1}^{k}\operatorname(\gamma_i).
\end{equation}

We can denote the length function of $\gamma$ on $\mathcal{M}_{g,m,n}$ as
\begin{equation}\notag
    \ell_{\gamma}(X)=\sum_{i=1}^{k}c_{i}\ell_{{\gamma}_{i}}(X),
\end{equation}
where $\ell_{{\gamma}_{i}}(X)$ is the length of ${\gamma}_{i}$ on $X\in \mathcal{M}_{g,m,n}$.

Given a function $f: \mathbb{R}_{+} \longrightarrow \mathbb{R}_{+}$,
\begin{equation}\notag
    f_{\gamma}(X)=\sum_{[\alpha] \in \mathcal{O}_{\gamma}}f(\ell_{\alpha}(X)),
\end{equation}
we can define a function $f_{\gamma}$:
\begin{equation}\notag
    f_{\gamma}:\mathcal{M}_{g,m,n} \longrightarrow \mathbb{R}_{+}.
\end{equation}
Now, we recall Mirzakhani's integral formula on $\mathcal{M}_{g,n}(\overrightarrow{L})$.
	
	\begin{theorem}\cite[Theorem 7.1]{MR2264808}
		For any multi-curve $\gamma=\sum_{i=1}^{k}c_i\gamma_i$, the integral of $f_{\gamma}$ over $\mathcal{M}_{g,n}(\overrightarrow{L})$ with respect to the Weil-Petersson volume form is given by
		\begin{equation}
			\notag\int_{\mathcal{M}_{g,n}(\overrightarrow{L})}f_{\gamma}(X)dX=\frac{2^{-M(\Gamma)}}{|Sym(\gamma)|}\int_{x\in\mathbb{R}^k_+}f(|x|)V_{g,n}(\Gamma,\overrightarrow x,\overrightarrow \beta,\overrightarrow L)\overrightarrow x\cdot d\overrightarrow x,
		\end{equation}
		where $\Gamma=(\gamma_1,\dots,\gamma_k)$, $|\overrightarrow x|=\sum_{i=1}^{k}c_ix_i$, $\overrightarrow x\cdot d\overrightarrow x=x_1\dots x_k\cdot dx_1\wedge\dots\wedge dx_k$ and
		\begin{equation}
			\notag	M(\Gamma):=|\{i|\gamma_i \text{ separates off a one-handle from } S_{g,n}\}|,
		\end{equation}
		and $V_{g,n}(\Gamma,\overrightarrow x,\overrightarrow \beta,\overrightarrow L)$ is defined by
		\begin{equation}
			\notag V_{g,n}(\Gamma,\overrightarrow x,\overrightarrow \beta,\overrightarrow L):=\operatorname{Vol}\left(\mathcal{M}\left(S_{g,n}(\gamma),\ell_{\Gamma}=\overrightarrow x,\ell_{\overrightarrow \beta}=\overrightarrow L\right)\right)=\prod_{i=1}^{s}V_{g_i,n_i}(\ell_{A_i}).
		\end{equation}
	\end{theorem}
    Based on the pants decomposition of the surface, we can generalize Theorem 7.1 of Mirzakhani's article \cite{MR2264808} to the case of cone surfaces. We first state it as following:
    \begin{lemma}\label{lem 3.2}
    For any multicurve $\gamma=\sum_{i=1}^{k}c_{i}\ell_{{\gamma}_{i}}(X)$, the integral of $f_{\gamma}$ over $\mathcal{M}_{g,m,n}$ with respect to the Weil-Petersson volume form is given by:
\begin{equation}\notag
    \int_{\mathcal{M}_{g,m,n}(\overrightarrow L,\overrightarrow\theta)}f_{\gamma}(X)dX=\frac{2^{-M(\Gamma)}}{|Sym(\gamma)|}\int_{\overrightarrow{x}\in \mathbb{R}^{k}_{+}}f(|\overrightarrow{x}|)V_{g,m,n}(\Gamma,\ \overrightarrow{x},\ \overrightarrow \beta,\ \overrightarrow L)\overrightarrow{x}d\overrightarrow{x},
\end{equation}
where $\Gamma=(\gamma_1,\cdots,\gamma_k)$, $|\overrightarrow{x}|=\sum_{i=1}^{k}c_{i}$, $\overrightarrow{x}\cdot d\overrightarrow{x}=x_1\cdots x_k \cdot dx_1\bigwedge\cdots\bigwedge dx_k$
and
\begin{equation}\notag
    M(\Gamma)=|\{i|\gamma_{i}\ \text{separates\ off\ a\ one-handle\ from}\ S_{g,m,n}\}|.
\end{equation}
    Also,

    \begin{equation}\notag
			\notag V_{g,m,n}(\Gamma,\overrightarrow x,\overrightarrow \beta,\overrightarrow L):=\operatorname{Vol}\left(\mathcal{M}\left(S_{g,m,n}(\gamma),\ell_{\Gamma}=\overrightarrow x,\ell_{\overrightarrow \beta}=\overrightarrow L\right)\right)=\prod_{i=1}^{s}V_{g_i,m_i,n_i}(\ell_{A_i}),
		\end{equation}
where
\begin{equation}\notag
    S_{g,m,n}(\gamma)=\bigcup_{i=1}^{s}S_i.
\end{equation}
    $S_i\cong S_{g_i,n_i}$, and $A_i=\partial S_{i}$.
\end{lemma}
Now we sketch how we extend the above formula over the moduli space of cone surfaces.

Fix a multi-curve $\gamma=\sum_{i=1}^{k}c_{i}\gamma_{i}\ (c_{i}\in \mathbb{R}_{+})$ on $S_{g,m,n}$. We denote $\Gamma=(\gamma_1,\cdots,\gamma_k).$

For an element $h\in Mod_{g,m,n},$ it acts on $\Gamma$ by
\begin{equation}\notag
    h\cdot\Gamma=(h\cdot\gamma_1,\cdots,h\cdot\gamma_k).
\end{equation}
Let $\mathcal{O}_\Gamma$ be the set of homotopic classes of elements of $Mod_{g,m,n}\cdot\Gamma$. And we let $\mathcal{M}$ be the set:
\begin{equation}\notag
    \mathcal{M}_{g,m,n}(\overrightarrow{L},\overrightarrow \theta)^{\Gamma}=\{(X,\eta)|X\in\mathcal{M}_{g,m,n}(\overrightarrow{L},\overrightarrow \theta),\eta=(\eta_1,\dots,\eta_k)\in\mathcal{O}_{\Gamma}\}.
\end{equation}

Thus we have:
\begin{equation}\notag
    \mathcal{M}_{g,m,n}(\overrightarrow L,\overrightarrow \theta)^{\Gamma}=\mathcal{T}_{g,n}(\overrightarrow L)/G_{\Gamma},
\end{equation}
where $G_{\Gamma}=\bigcap_{i=1}^{k}Stab(\gamma_i)\subset Mod_{g,m,n}$.

Now we define the covering map $\pi^{\Gamma}:\mathcal{M}_{g,m,n}(\overrightarrow L,\overrightarrow \theta)^{\Gamma}\longrightarrow\mathcal{M}_{g,m,n}(\overrightarrow L,\overrightarrow \theta)$ by
\begin{equation}\notag
    \pi^{\Gamma}(X,\eta):=X.
\end{equation}

Next we consider the twisting on $\Gamma.$ Let $\ell_{\Gamma}:\mathcal{T}_{g,m,n}\longrightarrow\mathbb{R}_{+}^{k}$ be:
\begin{equation}\notag
    \ell_{\Gamma}(X)=(\ell_{\gamma_1}(X),\cdots,\ell_{\gamma_k}(X)).
\end{equation}
Given $\overrightarrow x=(x_1,\cdots,x_k)\in\mathbb{R}_{+}^{k},$ the Weil-Petersson volume form induces a natural volume form on $\mathcal{M}_{g,m,n}(\overrightarrow L,\overrightarrow \theta)^{\Gamma}[\overrightarrow x]=\ell_{\Gamma}^{-1}(\overrightarrow x),$ for $\omega_{wp}$ is invariant under the mapping class group.

Let $\phi^{t}_{\gamma_i}$ be the twisting deformation along $\gamma_i$ with length $t$, which is given by $\phi_{\gamma_i}^{t}(X)=tw_{\gamma_i}^{t\cdot \ell_{\gamma_i}(X)}$. Then we can define the twisting on $\Gamma$ by
\begin{equation}\notag
\phi_{\gamma}^{(t_1,t_2,\cdots,t_k)}:=\prod_{i=1}^k\phi_{\gamma_i}^{t_i},
\end{equation}
which is a bijection from $\mathcal{T}_{g,m,n}$ to $\mathcal{T}_{g,m,n}$ and preserves the Weil-Petersson symplectic form $\omega_{wp}.$

On the other hand, by cutting $X\in \ell_{\Gamma}^{-1}(x_1,\cdots,x_k)$ along $\Gamma=(\gamma_1,\cdots,\gamma_{k}),$ we obtain a surface $s_\gamma(X)\in \mathcal{T}(S_{g,m,n}(\gamma),\ell_{\Gamma}=\overrightarrow x,\ell_{\beta}=\overrightarrow L)$ with geodesic boundary components as in (\ref{2.0}). Note that for any $(t_1,\cdots,t_k)\in \mathbb{R}_{+}^{k},$ we have:
\begin{align}
s_{\gamma}(X)=s_{\gamma}(\phi_{\gamma}^{(t_1,\cdots,t_k)}(X)).\label{3.1}
\end{align}
Then we consider the quotient space :
\begin{equation}
    \mathcal{M}_{g,m,n}(\overrightarrow L,\overrightarrow \theta)^{\Gamma*}[\overrightarrow x]=\mathcal{M}_{g,m,n}(\overrightarrow L,\overrightarrow \theta)[\overrightarrow x]^{\Gamma}/(T^k),\label{3.2}
\end{equation}
where $T^k=\prod^{k}_{i=1}S^1$, inherits a symplectic structure of $\mathcal{M}_{g,m,n}(\overrightarrow L,\overrightarrow \theta)$. 
By formula (\ref{3.1}), we know that $s_{\gamma}(X)$ is invariant under the twisting deformation $\phi_{\gamma}^{(t_1,\cdots,t_k)}$. Since the map $s_{\gamma}$ is mapping class group equivariant (which means that for any $h\in Mod_{g,m,n}$, we have $s_{\gamma}(h(X))=h(s_{\gamma}(X))$, it induces a map on $\mathcal{M}_{g,m,n}(\overrightarrow L,\overrightarrow \theta)^{\Gamma}$ as follows. For $(X,\eta)\in \mathcal{M}_{g,m,n}(\overrightarrow L,\overrightarrow \theta)^{\Gamma *}[a],$ let 
\begin{equation}\notag
    s((X,\eta))=s_{\gamma}(X)\in \mathcal{M}(S_{g,m,n}(\gamma),\ell_{\Gamma}=\overrightarrow x,L_{\beta}=\overrightarrow L).
\end{equation}
Now as in formula (\ref{2.1}), let 
\begin{equation}\notag
    S_{g,m,n}(\gamma)=\bigcup_{i=1}^{s}S_{g_i,m_i,n_i}, \ A_{i}=\partial S_{g_i,m_i,n_i}\in\partial(S_{g,m,n}(\gamma)).
\end{equation}
Thus by formula (\ref{2.2}) and formula (\ref{3.2}), we have the following lemma:
\begin{lemma} \label{lem 3.3}
    For any multi-curve $\gamma$, the canonical isomorphism:
    \begin{equation} \notag
        s:\mathcal{M}_{g,m,n}(\overrightarrow L,\overrightarrow \theta)^{\Gamma *}[x]\longrightarrow \mathcal{M}(S_{g,m,n}(\gamma),\ell_{\Gamma}=\overrightarrow x,L_{\beta}=\overrightarrow L)\cong \prod_{i=1}^{s}\mathcal{M}_{g_i,m_i,n_i}(\ell_{A_i})
    \end{equation}
    is a symplectomorphism.
\end{lemma}
The symplectomorphism means that the symplectic form of $\mathcal{M}_{g,m,n}(\overrightarrow L,\overrightarrow \theta)^{\Gamma *}[a]$ which inherits from the $\mathcal{M}_{g,m,n}(\overrightarrow L,\overrightarrow \theta)$ is equal to the pull-back of the Weil-Petersson symplectic on $\prod_{i=1}^{s}\mathcal{M}_{g_i,m_i,n_i}(\ell_{A_i}).$

Then we can get the integration formula over the covering space $\mathcal{M}_{g,m,n}(\overrightarrow L,\overrightarrow \theta)^{\Gamma}$.
\begin{lemma}
    For any function $F:\mathbb{R}^k\longrightarrow \mathbb{R}_{+}$ and $\Gamma=(\gamma_1,\cdots,\gamma_k)$,
    define $F_{\Gamma}:\mathcal{M}_{g,m,n}(\overrightarrow L,\overrightarrow\theta)\longrightarrow \mathbb{R}_{+}$ by
\begin{equation}\notag
     F_{\Gamma}(Y)=F(\ell_{\Gamma}(Y)).
\end{equation}
         
    Then the integral of $F_{\Gamma}$ over $\mathcal{M}_{g,m,n}(\overrightarrow L, \overrightarrow \theta)$ is given by
\begin{equation}\notag
     \int_{\mathcal{M}_{g,m,n}(\overrightarrow L, \overrightarrow \theta)}F_{\Gamma}(Y)dY=2^{-M(\Gamma)}\int_{\overrightarrow x \in \mathbb{R}_{+}^{k}}F(\overrightarrow x)V_{g,m,n}(\Gamma,\ \overrightarrow{x},\ \overrightarrow \beta,\ \overrightarrow L)\overrightarrow xd\overrightarrow x,
\end{equation}
   
where $\overrightarrow x\cdot d\overrightarrow x=x_1\dots x_k\cdot dx_1\wedge\dots\wedge dx_k$ and
$\overrightarrow x=(x_1,\cdots,x_k),$ and 
\begin{equation}\notag
    M(\Gamma):=|\{i|\gamma_i \text{ separates off a one-handle from } S_{g,n}\}|.
\end{equation}
\end{lemma}
The discussion of the definition of $M(\Gamma)$ is the same as page 214 of \cite{MR2264808}, so we omit it and use the definition in Theorem 7.1 of \cite{MR2264808}.
\begin{proof}
Fix $x\in \mathbb{R}^k_{+}$, we consider the set
$\ell_{\Gamma}^{-1}(\overrightarrow x)$, where the function $F_{\Gamma}=F(\ell_{\Gamma}(X))$ is a constant on this set.
For $x\in \mathbb{R}_{+}$, let:
\begin{equation}\notag
    I(\overrightarrow x)=\int_{\mathcal{M}_{g,m,n}(\overrightarrow L,\overrightarrow \theta)^{\Gamma}[\overrightarrow x]}F(\ell_{\Gamma}(Y))dY.
\end{equation}

By Theorem \ref{thm 2.6},
\begin{equation}\notag
    \int_{\mathcal{M}_{g,m,n}^{\Gamma}(\overrightarrow L,\overrightarrow \theta)}F_{\Gamma}(Y)dY=\int_{\overrightarrow x \in \mathbb{R}_{+}^{k}}I(\overrightarrow x)d\overrightarrow x.
\end{equation}

To calculate $I(\overrightarrow x)$, we need to show:
\begin{equation}\notag
    \operatorname{Vol}(l_{\Gamma}^{-1}(\overrightarrow x))=2^{-M(\Gamma)}x_1\cdots x_kV_{g,m,n}(\Gamma,\ \overrightarrow{x},\ \overrightarrow \beta,\ \overrightarrow L).
\end{equation}

We consider the projection map:
\begin{equation}\notag
    \pi:\mathcal{M}_{g,m,n}(\overrightarrow L,\overrightarrow \theta)[\overrightarrow x]^{\Gamma}\longrightarrow\mathcal{M}_{g,m,n}(\overrightarrow L,\overrightarrow \theta)^{\Gamma*}[\overrightarrow x].
\end{equation}
Combining with the discussion of $M(\Gamma)$ in page 214 in \cite{MR2264808}, for an open set $U\subset\mathcal{M}_{g,m,n}(\overrightarrow L,\overrightarrow \theta)^{\Gamma*}[\overrightarrow x]$, we have:
\begin{equation}\notag
    \operatorname{Vol}(\pi^{-1}(U))=2^{-M(\Gamma)}\operatorname{Vol}(U)\cdot x_1\cdots x_k.
\end{equation}

Thus by Lemma \ref{lem 3.3}:
\begin{equation}
    \notag
I(\overrightarrow x)=2^{-M(\Gamma)}\cdot F(\overrightarrow x)x_1\cdots x_k\cdot V_{g,m,n}(\Gamma,\ \overrightarrow{x},\ \overrightarrow \beta,\ \overrightarrow L).
\end{equation}
\end{proof}
Regard the $\mathcal{M}_{g,m,n}(\overrightarrow L,\overrightarrow \theta)^{\Gamma}$ as a covering space of $\mathcal{M}_{g,m,n}(\overrightarrow L,\overrightarrow \theta)$ with covering number $|Sym(\gamma)|$, we get the Lemma \ref{lem 3.2}.
\subsection{The volume of $\mathcal{M}_{1,0,1}(\theta)$}
In this part, we consider the volume of the moduli space of a compact hyperbolic surface with one cone point and one genus without boundaries. We denote this surface by $S_{1,0,1}(\theta)$, and denote its moduli space by $\mathcal{M}_{1,0,1}(\theta)$. The volume of the moduli space $\mathcal{M}_{1,0,1}(\theta)$ is denoted by $V_{1,0,1}(\theta)$.

The main idea for integrating over $\mathcal{M}_{1,0,1}(\theta)$ is the spliting along a simple closed curve $\gamma$ on the surface. In this case, we can divide the surface into simpler surfaces. And the moduli space of the original surface can be expressed as a product of the moduli spaces of the simpler surfaces. We cut the surface $S_{1,0,1}(\theta)$ along a simple closed curve $\gamma$. Then we obtain a pair of pants with one cone point and two geodesic boundaries. We define $\mathcal{M}(S_{0,2,1},\ell_{\gamma}=t,\theta)$ to be the moduli space of Riemann surfaces homeomorphic to $S_{0,2,1}(\gamma,\theta)$ such that the lengths of the two boundary components corresponding to $\gamma$ are equal to $t$.

By the discussion of Section \ref{subsec 3.1}, we establish the following integration lemma on $\mathcal{M}_{1,0,1}(\theta)$:
\begin{lemma}
    For any function $F:\mathbb{R}\longrightarrow \mathbb{R}_{+}$ and a closed geodesic $\gamma$,
    define $F_{\gamma}:\mathcal{M}_{1,0,1}(\theta)\longrightarrow \mathbb{R}$ by
\begin{equation}\notag
     F_{\gamma}(Y)=F(\ell_{\gamma}(Y)).
\end{equation}
         
    Then the integral of $F_{\gamma}$ over $\mathcal{M}_{1,0,1}(\theta)$ is given by
\begin{equation}\notag
    \int_{\mathcal{M}_{1,0,1}(\theta)}F_{\gamma}(Y)dY=\int_{0}^{\infty}F(x)\operatorname{Vol}(\mathcal{M}(S_{0,2,1}(\gamma),\theta,\ell_\gamma=x)xdx.
\end{equation}
    In particular, $\operatorname{Vol}(\mathcal{M}(S_{0,2,1}(\gamma),\theta,\ell_\gamma=x)=1.$

\end{lemma}
Thus we can calculate $V_{1,0,1}(\theta).$
\begin{theorem}
\begin{equation} \notag
     V_{1,0,1}(\theta)=-\frac{\theta^2}{48}+\frac{\pi^2}{12}.
\end{equation}
    \end{theorem}
    \begin{proof}
We elaborate the main idea of calculation of $V_{g,m,n}$ through an example of g=n=1.

We need to use the generalized McShane's identity in Theorem \ref{thm 2.4}. That is
\begin{equation}\notag
    \sum 2\tan^{-1}(\frac{\sin{\frac{\theta}{2}}}{\cos{\frac{\theta}{2}}+exp^\frac{|\alpha|+|\beta|}{2}})=\frac{\theta}{2}.
\end{equation}
 As we know
\begin{equation}\notag
    \cos{\theta}=\frac{e^{i\theta}+e^{-i\theta}}{2},\ \sin{\theta}=\frac{e^{i\theta}-e^{-i\theta}}{2i},\ \tan{\theta}=\frac{e^{i\theta}-e^{-i\theta}}{ie^{i\theta}+ie^{-i\theta}},\ \tan^{-1}{z}=\frac{\ln{\frac{i-z}{i+z}}}{2i}.
\end{equation}
We have:
\begin{equation}\notag
     \tan^{-1}(\frac{\sin{\frac{\theta}{2}}}{\cos{\frac{\theta}{2}}+exp(\frac{y+z}{2})})=
 \frac{1}{2i}\ln\frac{i(\exp({\frac{y+z}{2}})+{\frac{e^{\frac{i\theta}{2}}+e^\frac{-i\theta}{2}}{2}})-{\frac{e^{\frac{i\theta}{2}}-e^\frac{-i\theta}{2}}{2i}}}{i(\exp({\frac{y+z}{2}})+{\frac{e^{\frac{i\theta}{2}}+e^\frac{-i\theta}{2}}{2}})+\frac{e^\frac{-i\theta}{2}-e^\frac{-i\theta}{2}}{2i}}.
\end{equation}

 This implies that:
 \begin{equation}\notag
      \sum\frac{1}{i}\ln\frac{e^{\frac{i\theta}{2}}+e^\frac{y+z}{2}}{e^\frac{-i\theta}{2}+e^\frac{y+z}{2}}=\frac{\theta}{2},
 \end{equation}
 where $y,\ x$ is the length of the simple closed geodesics on M which bound with the cone point an embedded pair of pants. Denote the left side of the above identity by $\mathcal{D}(\theta,y,x)$. Then we can calculate the partial derivatives of $\mathcal{D}(\theta,y,x).$

In this case, $y=z$. Then we have:
\begin{equation}\notag
    2\frac{\partial \mathcal{D}(\theta,x,x)}{\partial {\theta}}=\frac{1}{1+e^{x-\frac{i\theta}{2}}}+\frac{1}{1+e^{x+\frac{i\theta}{2}}}.
\end{equation}

Integrating over the above generalized McShane's identity over $\mathcal{M}_{1,0,1}(\theta)$, we obtain:
\begin{equation}
    \notag
\theta \cdot V_{1,0,1}(\theta)=\int_{0}^{\infty} x \mathcal{D}(\theta,x,x) dx.
\end{equation}

So we have:
\begin{equation}\notag
    \frac{\partial (\theta \cdot V_{1,0,1}(\theta))}{\partial \theta}=\frac{1}{2}\int_{0}^{\infty} x (\frac{1}{1+e^{x-\frac{i\theta}{2}}}+\frac{1}{1+e^{x+\frac{i\theta}{2}}}) dx.
\end{equation}

In order to calculate this integration, we need to integrate it in the complex plane $\mathbb{C}$. By setting $y_1=x+\frac{i\theta}{2},\ y_2=x-\frac{i\theta}{2}$, we have:
\begin{equation}\notag
    \int_{0}^{\infty} x (\frac{1}{1+e^{x-\frac{i\theta}{2}}}+\frac{1}{1+e^{x+\frac{i\theta}{2}}}) dx=\int_{0+\frac{i\theta}{2}}^{\infty+\frac{i\theta}{2}} (y_1-\frac{i\theta}{2})\frac{1}{1+e^{y_1}}+\int_{0-\frac{i\theta}{2}}^{\infty-\frac{i\theta}{2}} (y_2+\frac{i\theta}{2})\frac{1}{1+e^{y_2}}.
\end{equation}

As $\theta<\pi$, by residue theorem, we have:
\begin{align}
\label{A1}\notag
         &\int_{0+\frac{i\theta}{2}}^{\infty+\frac{i\theta}{2}} (y_1-\frac{i\theta}{2})\frac{1}{1+e^{y_1}}d{y_1}\\
=&\int_{0}^{\infty} (y_1-\frac{i\theta}{2})\frac{1}{1+e^{y_1}}dy_{1}
+\int_{\frac{i\theta}{2}}^{0}(y_1-\frac{i\theta}{2})\frac{1}{1+e^{y_1}}dy_{1}\\
\notag
+&\int_{\infty}^{\infty+\frac{i\theta}{2}} (y_1-\frac{i\theta}{2})\frac{1}{1+e^{y_1}}dy_{1}.    
\end{align}
    \begin{align}
    \label{A2}\notag
         &\int_{0-\frac{i\theta}{2}}^{\infty-\frac{i\theta}{2}} (y_2+\frac{i\theta}{2})\frac{1}{1+e^{y_2}}dy_{2}\\
     =&\int_{0}^{\infty} (y_2+\frac{i\theta}{2})\frac{1}{1+e^{y_2}}dy_{2}
    +\int_{-\frac{i\theta}{2}}^{0} (y_2+\frac{i\theta}{2})\frac{1}{1+e^{y_2}}dy_{2}\\
    \notag+&\int_{\infty}^{\infty-\frac{i\theta}{2}} (y_2+\frac{i\theta}{2})\frac{1}{1+e^{y_2}}dy_{2}.
    \end{align}

In the two formulas, the last terms can be controlled by $\epsilon$. So they will tends to 0 as $y\longrightarrow \infty$. Then we only need to calculate the front two terms.
Add the equation $(\ref{A1})$ to the equation $(\ref{A2})$, we have:
\begin{align}
\notag
    &\int_{0}^{\infty} x (\frac{1}{1+e^{x-\frac{i\theta}{2}}}+\frac{1}{1+e^{x+\frac{i\theta}{2}}})dx\\
\notag
=&\int_{\frac{i\theta}{2}}^{0} (y_1-\frac{i\theta}{2})\frac{1}{1+e^{y_1}}dy_{1}+\int_{-\frac{i\theta}{2}}^{0} (y_2+\frac{i\theta}{2})\frac{1}{1+e^{y_2}}dy_{2}\\
\notag
+&2\int_{0}^{\infty} y_2\frac{1}{1+e^{y_2}}dy_{2}.
\end{align}

By the special case $Li_2(-1)$ in page 6 of \cite{MR2290758}, we know that:

$$2\int_{0}^{\infty} y_2\frac{1}{1+e^{y_2}}dy_{2}=\frac{\pi^2}{6}.$$

To calculate the first two terms, we set $y_1=ix_1,\ y_2=ix_2$. Then we get:
\begin{align}
\notag
    &\int_{0}^{\infty} x (\frac{1}{1+e^{x-\frac{i\theta}{2}}}+\frac{1}{1+e^{x+\frac{i\theta}{2}}})dx\\
\notag
=&\frac{\pi^2}{6}-\int_{-\frac{\theta}{2}}^{0} (x_2+\frac{\theta}{2})\frac{1}{1+e^{ix_2}}dx_{2}-\int_{\frac{\theta}{2}}^{0} (x_1-\frac{\theta}{2})\frac{1}{1+e^{ix_1}}dx_{1}\\
\notag
=&\frac{\pi^2}{6}-\int_{\frac{\theta}{2}}^{0}(x-\frac{\theta}{2})(\frac{1}{1+e^{ix}}+\frac{1}{1+e^{-ix}})dx.
\end{align}

Since we have
\begin{equation}\notag
    \frac{1}{1+e^{ix}}+\frac{1}{1+e^{-ix}}=1.
\end{equation}

Therefore, we have:
\begin{equation}\notag
    V_{1,1}(\theta)=\frac{1}{2}(-\frac{\theta^2}{24}+\frac{\pi^2}{6})=-\frac{\theta^2}{48}+\frac{\pi^2}{12},
\end{equation}

which is the evaluation of $i\theta$ for Mirzakhani's polynomial of $g=1,\ n=1$ in page 207 of \cite{MR2264808}.
\end{proof}

\section{The recursion of the volume polynomials}\label{sec4}
The discussion is established on a compact hyperbolic surface $M$ with all cone points in $(0,\pi]$. According to the definitions in pages 83-85 of \cite{MR2215456}, fixing $\gamma$ as a cone point of angle $\theta$ or a geodesic boundary component of length $L$, we can define the Gap functions of different situations by:

Case 1. Both $\alpha$ and $\beta$ are a simple closed geodesic in the geometric interior of $M$ which bound a pair of pants with $\gamma$. In this case:
\begin{equation}\notag
\operatorname{Gap}_1(\gamma;\alpha,\beta))=2\tanh^{-1}\left(\frac{\sinh\frac{|\gamma|}{2}}{\cosh\frac{|\gamma|}{2}+\exp\frac{|\alpha|+|\beta|}{2}}\right).
\end{equation}

Case 2. $\alpha$ is a boundary geodesic and $\beta$ is a simple closed geodesic in the geometric interior of $M$. In this case:
\begin{equation}
\notag
\operatorname{Gap}_2(\gamma;\alpha,\beta))=|\frac{|\gamma|}{2}|-\tanh^{-1}\left(\frac{\sinh\frac{|\gamma|}{2}\sinh{\frac{|\beta|}{2}}}{\cosh\frac{|\alpha|}{2}+\cosh{\frac{|\gamma|}{2}}\cosh\frac{|\beta|}{2}}\right).
\end{equation}

Case 3. $\alpha$ is a cone point of angle $\phi$ and $\beta$ is  a simple closed geodesic in the geometric interior of $M$. In this case:
\begin{equation}\notag
\operatorname{Gap}_3(\gamma;\alpha,\beta))=|\frac{|\gamma|}{2}|-\tan^{-1}\left(\frac{\sin\frac{|\gamma|}{2}\sinh{\frac{|\beta|}{2}}}{\cos\frac{\phi}{2}+\cos{\frac{|\gamma|}{2}}\cosh\frac{|\beta|}{2}}\right).
\end{equation}

Case 4. $\alpha$ is a cusp and $\beta$ is  a simple closed geodesic in the geometric interior of M. In this case:
\begin{equation}\notag
    \operatorname{Gap}_4(\theta;\alpha,\beta))=|\frac{|\gamma|}{2}|-\tan^{-1}\left(\frac{\sin\frac{\gamma}{2}\sinh{\frac{|\beta|}{2}}}{1+\cos{\frac{\gamma}{2}}\cosh\frac{|\beta|}{2}}\right).
\end{equation}

If $\gamma$ is a geodesic boundary component of length $L$, then $|\gamma|=L$; if $\gamma$ is a cone point with angle $\theta$, then $|\gamma|=i\theta$. 
According to Theorem 1.8 in \cite{MR2215456}, we have:
\begin{theorem}
For a compact hyperbolic surface $S_{g,m,n}$ and fix a cone point $\theta$, we have that:
\begin{equation}\notag
    \sum_{\alpha,\beta}\operatorname{Gap}(\theta;\alpha,\beta)=\frac{\theta}{2},
\end{equation}
where $\alpha$, $\beta$ satisfies that:

1. $\alpha$, $\beta$ are simple geodesics or cusps or cone points.

2. $\alpha$, $\beta$ bound with $\theta$ an embedded pair of pants on $M$.

\end{theorem}
Because the Gap functions satisfy the definition of $f_{\gamma}$ in Section \ref{subsec 3.1}, so we can integrate the partial of the function $\operatorname{Gap}_{1}$ and $\operatorname{Gap}_{2}$ over the moduli space by Lemma \ref{lem 3.2}. So we can integrate both sides of the McShane's identity over the moduli space of $S_{g,m,n}$.
If we consider a cone point as a geodesic boundary with length  $i\theta$, these formulas are generalizations of Mirzakhani's function $\mathcal{R}$ and $\mathcal{D}$ in Lemma 3.1 \cite{MR2264808}. Because of the existence of the pants decomposition, we can generalized Mirzakhani's recursion formula to compact hyperbolic surface with cone points. We will clarify this extension later.

In the calculation of $V_{1,0,1}(\theta)$, we already used the above idea. The only difference for our discussion and that of Mirzakhani is the Gap functions.

Then we have the following theorem:

    \begin{theorem}
     For a compact hyperbolic surfaces $S_{g,m,n}(L,i\theta)$ with cone points for all angles of cone in $(0,\pi]$, there exists a recursion formula for the             $V_g(\ell_1,\cdots,\ell_m,\\
     i\theta_1,\cdots,i\theta_n)$. That is:
     \begin{align}
          \notag &\frac{\partial({\frac{1}{2}{\theta_1 V_{g,m,n}(\ell_1,\cdots,\ell_m,\theta_1,\theta_2,\cdots,\theta_{n})}})}{\partial{\theta_1}}\\
\notag
=&\int_0^{\infty}\int_0^{\infty}\frac{\partial {\operatorname{Gap}_{1}(\theta_1,x,y)}}{\partial \theta_1}\cdot\frac{1}{2^{\delta(g-1,m+2,n-1)}}\cdot V_{g-1,m+1,n-1}(x,y,\ell_1,\cdots,\ell_{m},\\\notag&\theta_2,\cdots,\theta_n)dxdy\\
\notag
+&\int_0^{\infty}\int_0^{\infty}\frac{\partial {\operatorname{Gap}_{1}(\theta_1,x,y)}}{\partial \theta_1}\sum_{a \in I_{g,n,m}}V(a,x,y,\ell_{1},\cdots,\ell_{m},\theta_2,\cdots,\theta_n)dxdy\\
\notag
+&\int_0^{\infty}\frac{\partial {\operatorname{Gap}_{2}(\theta_1,x,y)}}{\partial \theta_1}\cdot\frac{1}{2^{\delta(g,n-1)}}\sum_{j=1}^{m}V_{g,m,n-1}(x,\ell_{1},\cdots,\ell_{j-1},\ell_{j+1},\cdots,\ell_{m},\\\notag&\theta_2,\cdots\theta_n)dx\\
\notag
+&\int_0^{\infty}\frac{\partial {\operatorname{Gap}_{3}(\theta_1,x,\ell_j)}}{\partial \theta_1}\cdot\frac{1}{2^{\delta(g,m+1,n-2)}}\sum_{j=1}^{m}V_{g,m+1,n-2}(x,\ell_{1},\cdots,\ell_{j-1},\ell_{j+1},\cdots,\ell_{m},\\\notag&\theta_2,\cdots\theta_n)dx.
     \end{align}
     And we have:
\begin{align}
   \notag &\frac{\partial({\frac{1}{2}{\ell_{m+1} V_{g,m+1,n-1}(\ell_1,\cdots\ell_{m+1},\theta_2,\cdots,\theta_{n})}})}{\partial{\ell_{m+1}}}\\
\notag
=&\int_0^{\infty}\int_0^{\infty}\frac{\partial {\operatorname{Gap}_{1}(\ell_1,x,y)}}{\partial \ell_1}\cdot\frac{1}{2^{\delta(g-1,m+2,n-1)}}\cdot V_{g-1,m+2,n-1}(x,y,\ell_1,\cdots,\ell_{m},\\\notag&\theta_2,\cdots,\theta_n)dxdy\\
\notag
+&\int_0^{\infty}\int_0^{\infty}\frac{\partial {\operatorname{Gap}_{1}(\ell_1,x,y)}}{\partial \ell_1}\sum_{a \in I_{g,m+1,n-1}}V(a,x,y,\ell_{1},\cdots,\ell_{m},\theta_2,\cdots,\theta_n)dxdy\\
\notag
+&\int_0^{\infty}\frac{\partial {\operatorname{Gap}_{2}(\ell_1,x,\ell_j)}}{\partial \ell_1}\cdot\frac{1}{2^{\delta(g,m,n-1)}}\sum_{j=2}^{m}V_{g,m,n-1}(x,\ell_{1},\cdots,\ell_{j-1},\ell_{j+1},\cdots,\ell_{m},\\\notag&\theta_2,\cdots\theta_n)dx\\
\notag
+&\int_0^{\infty}\frac{\partial {\operatorname{Gap}_{3}(\ell_1,\theta_1,\theta_j)}}{\partial \ell_1}\cdot\frac{1}{2^{\delta(g,m+1,n-2)}}\sum_{i=1}^{n}V_{g,m+1,n-2}(x,\ell_{1},\cdots,\ell_{m},\\\notag&\theta_2,\cdots,\theta_{i-1},\theta_{i+1}\cdots,\theta_n)dx,
\end{align}

where $\delta(g,m,n)=0$ unless $g=1,m=1,n=0.$ When $g=1,\ m=1,\ n=0,$ we have $\delta(g,m,n)=1.$

\end{theorem}
\begin{proof}
  We will clarify the generalization of Mirzakhani's recursion formula in Theorem 8.1 of \cite{MR2264808} by induction.
  First, we consider the surface $S_{g,m,1}(\overrightarrow L, \theta)$. If the cone point is replaced by a geodesic boundary with length $\ell_1$ (which means that we consider the surface $S_{g,m+1}(\ell_1,\overrightarrow L)$, we have the following formula, which is expressed in Theorem 8.1 in page 218 of \cite{MR2264808}. The definitions of the three functions in Theorem 8.1 of \cite{MR2264808} are in pages 204-205 of \cite{MR2264808}, we directly use the expression of the three functions here:
\begin{align}
   \notag &\frac{\partial(\frac{1}{2}{\ell_1V_{g}(\ell_1,\cdots,\ell_n)})}{\partial{\ell_1}}\\
\notag=&\int_0^{\infty}\int_0^{\infty}\frac{\partial {\operatorname{Gap}}(\ell_1,x,y)}{\partial \ell_1}\cdot\frac{1}{2^{\delta(g-1,n+1)}}\cdot V_{g-1,n+1}(x,y,\ell_{2},\cdots,\ell_{n})dxdy\\
\notag
+&\int_0^{\infty}\int_0^{\infty}\frac{\partial {\operatorname{Gap}_1}(\ell_1,x,y)}{\partial \ell_1}\cdot\frac{1}{2^{\delta(g,n-1)}}\sum_{a\in I_{g,n}}V(a,x,y,\ell_{2},\cdots,\ell_{n})dxdy\\
\notag
+&\int_0^{\infty}\frac{\partial {\operatorname{Gap}_2}(\ell_1,\ell_j,x)}{\partial \ell_1}\cdot\frac{1}{2^{\delta(g,n-1)}}\sum_{j=2}^{n}V_{g,n-1}(x,\ell_{2},\cdots,\ell_{j-1},\ell_{j+1},\cdots,\ell_{n})dx.
\end{align}

 In terms of one cone point, we have the following formula by splitting the surface. The splitting of the cone surface is clarified in Figure \ref{Figure 1}:

\begin{figure}[!htbp]
\centering 
\includegraphics[width=1.0\textwidth,trim=0 450 0 0]{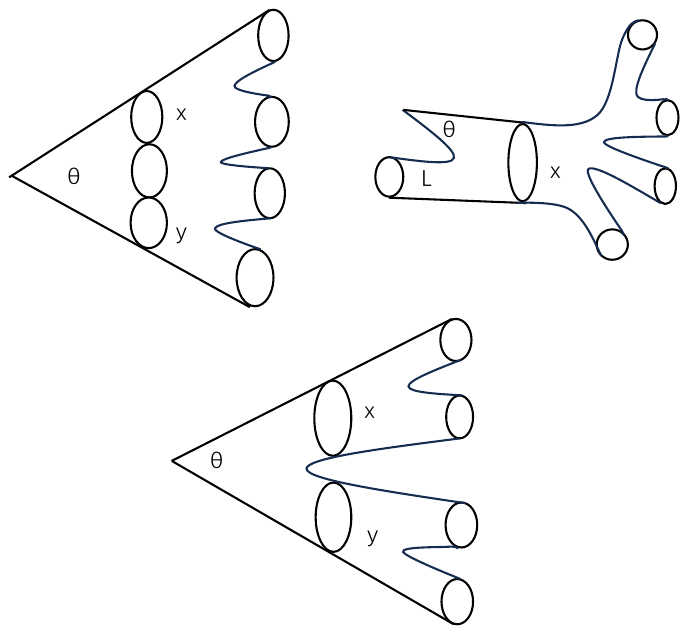} 
\caption{The spliting of n cone points} 
\label{Figure 1}
\end{figure}

\begin{align}\notag
&\frac{\partial(\frac{1}{2}{\theta V_{g,n-1,1}(\overrightarrow L,\overrightarrow \theta)})}{\partial{\theta}}\\
\notag
=&\int_0^{\infty}\int_0^{\infty}\frac{\partial {\operatorname{Gap}_{1}(\theta,x,y)}}{\partial \theta}\cdot\frac{1}{2^{\delta(g-1,n+1)}}\cdot V_{g-1,n+1}(x,y,\ell_{2},\cdots,\ell_{n})dxdy\\
\notag+&\int_0^{\infty}\int_0^{\infty}\frac{\partial {\operatorname{Gap}_{1}(\theta,x,y)}}{\partial \theta}\sum_{a \in I_{g,n}}V(a,x,y,\ell_{2},\cdots,\ell_{n})dxdy\\
\notag+&\int_0^{\infty}\frac{\partial {\operatorname{Gap}_{2}(\theta,\ell_j,x)}}{\partial \theta}\cdot\frac{1}{2^{\delta(g,n-1)}}\sum_{j=2}^{n}V_{g,n-1}(x,\ell_{2},\cdots,\ell_{j-1},\ell_{j+1},\cdots,\ell_{n})dx,
  \end{align}
where the Gap functions are given in the start of this section. And we have:

$$\frac{\partial {\operatorname{Gap}_{1}}}{\partial \theta}(\theta,x,y)=\frac{\partial {\operatorname{Gap}_1}}{\partial \ell_1}(i\theta,x,y),$$

and

$$\frac{\partial {\operatorname{Gap}_{2}}}{\partial \theta}(\theta,x,y)=\frac{\partial {\operatorname{Gap}_2}}{\partial \ell_1}(i\theta,x,y).$$

Comparing with the formula of the geodesic boundary case, we get the formula for the surface with only one cone point, which is an evaluation of the Mirzakhani's polynomial at $\ell_n=i\theta$:

$$V_{g,n-1,1}(\ell_1,\cdots,\ell_{n-1},\theta)=V_{g,n}(\ell_1,\cdots,\ell_n)|_{\ell_n=i\theta},$$
which shows that our conclusion is true for the situation of one cone point.

Now we consider the case that the compact hyperbolic surface with $n$ cone points. We will prove the conclusion by induction.

$$Induction$$

Suppose that the formula is true for $k$ cone points case, $1\le k\le n-1$.

Let $I_{g,m,n}$ be the set of ordered pairs:

$$a=((g_1,I_1,J_1),(g_2,I_2,J_2)),$$

where $I_1,I_2 \subset \{1,\cdots,m\},\ J_1,J_2 \subset \{2,\cdots,n\}$ such that:

1. the two sets $I_1,I_2$ are disjoint and their union is $\{1,\cdots,m\}.$

2. the two sets $J_1,J_2$ are disjoint and their union is $\{2,\cdots,n\}.$

3. the numbers $g_1,g_2\geq0$ and $n_1=|I_1|,n_2=|I_2|,m_1=|J_1|,m_2=|J_2|$ satisfy:
\begin{align}\notag
2\leq 2g_1+n_1+m_1,\\
\notag
2\leq 2g_2+n_2+m_2,
\end{align}

and
\begin{equation}\notag
    g_1+g_2=g.
\end{equation}

Given
\begin{equation}\notag
    a=((g_1,I_1),(g_2,I_2))\in I_{g,n,m}.
\end{equation}
For $I=\{i_1,\cdots,i_k\}$, define $\ell_{I}$ by:
\begin{equation}\notag
    \ell_{I}=(\ell_{i_1},\cdots,\ell_{i_k}).
\end{equation}

For $J=\{j_1,\cdots,j_k\}$, define $\theta_{J}$ by:
\begin{equation}\notag
    \theta_{J}=(\ell_{j_1},\cdots,\ell_{j_k}).
\end{equation}
Let
\begin{equation}
V(a,x,y,\ell_1,\cdots,\ell_m,\theta_2,\cdots,\theta_m)=\frac{V_{g_1,n_1+1,m_1}(x,\ell_{I_1},\theta_{J_1})\times V_{g_2,n_2+1,m_2}(x,\ell_{I_2},\theta_{J_2})}{{2^{\delta(g_1,m_1+1,n_1)}}{2^{\delta(g_2,m_2+1,n_2)}}}.
\end{equation}

We  split the surface as in Figure \ref{Figure 2} .
\begin{figure}[!htbp]
\centering 
\includegraphics[width=1.0\textwidth,trim=0 450 0 0]{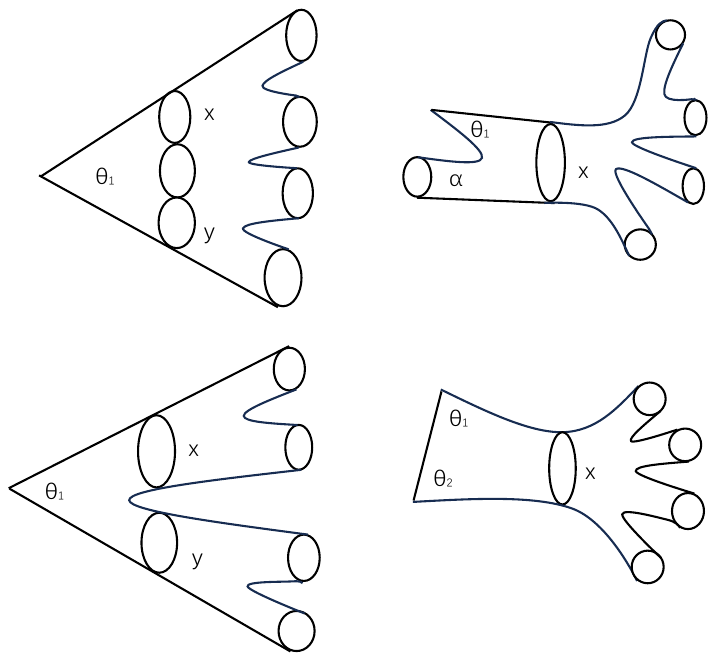} 
\caption{The spliting of n cone points} 
\label{Figure 2}

And we use the McShane's identity to obtain the following equation $(\ref{B1})$:

\end{figure}
\begin{align}
\notag
    &\frac{\partial({\frac{1}{2}{\theta_1 V_{g,m,n}(\ell_1,\cdots,\ell_m,\theta_1,\theta_2,\cdots,\theta_{n})}})}{\partial{\theta_1}}\\
\notag
=&\int_0^{\infty}\int_0^{\infty}\frac{\partial {\operatorname{Gap}_{1}}}{\partial \theta_1}\cdot\frac{1}{2^{\delta(g-1,m+2,n-1)}}\cdot V_{g-1,m+2,n-1}(x,y,\ell_1,\cdots,\ell_{m},\\\notag&\theta_2,\cdots,\theta_n)dxdy\\
\label{B1}
+&\int_0^{\infty}\int_0^{\infty}\frac{\partial {\operatorname{Gap}_{1}}}{\partial \theta_1}\sum_{a \in I_{g,n,m}}V(a,x,y,\ell_{1},\cdots,\ell_{m},\theta_2,\cdots,\theta_n)dxdy\\
\notag
+&\int_0^{\infty}\frac{\partial {\operatorname{Gap}_{2}}}{\partial \theta_1}\cdot\frac{1}{2^{\delta(g,n-1)}}\sum_{j=1}^{m}V_{g,n,m-1}(x,\ell_{1},\cdots,\ell_{j-1},\ell_{j+1},\cdots,\ell_{n},\theta_2,\cdots\theta_n)dx\\
\notag
+&\int_0^{\infty}\frac{\partial {\operatorname{Gap}_{3}}}{\partial \theta_1}\cdot\frac{1}{2^{\delta(g,n+1,m-2)}}\sum_{j=1}^{m}V_{g,n+1,m-2}(x,\ell_{1},\cdots,\ell_{j-1},\ell_{j+1},\cdots,\ell_{n},\\\notag&\theta_2,\cdots\theta_n)dx.
\end{align}

Substituting the cone point by a geodesic boundary with length $\ell_{m+1}$, which means that we consider the Riemann surface $S_{g,m+1,n-1}(l_1,\cdots,l_{m+1},\theta_2,\cdots,\theta_n)$ we have the following equation $(\ref{B2})$, which was shown in Theorem 8.1 in  \cite{MR2264808}, the splitting of this case is as in Figure \ref{Figure 3}:

\begin{figure}[!htbp]
\centering 
\includegraphics[width=1.0\textwidth,trim=0 450 0 0]{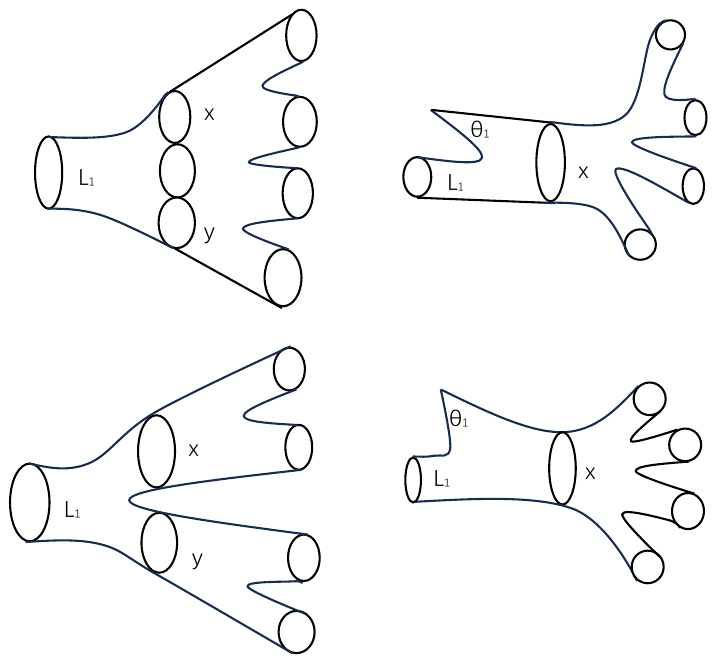} 
\caption{The Spliting of 1 geodesic boundary and n-1 cone points} 
\label{Figure 3}
\end{figure}
\begin{align}\notag
    &\frac{\partial({\frac{1}{2}{\ell_{m+1} V_{g,m+1,n-1}(\ell_1,\cdots\ell_{m+1},\theta_2,\cdots,\theta_{n})}})}{\partial{\ell_{m+1}}}\\
\notag =&\int_0^{\infty}\int_0^{\infty}\frac{\partial {\operatorname{Gap}_{1}}}{\partial \ell_1}\cdot\frac{1}{2^{\delta(g-1,m+2,n-1)}}\cdot V_{g-1,m+2,n-1}(x,y,\ell_1,\cdots,\ell_{m},\\\notag&\theta_2,\cdots,\theta_n)dxdy\\
\label{B2}
+&\int_0^{\infty}\int_0^{\infty}\frac{\partial {\operatorname{Gap}_{1}}}{\partial \ell_1}\sum_{a \in I_{g,m+1,n-1}}V(a,x,y,\ell_{1},\cdots,\ell_{m},\theta_2,\cdots,\theta_n)dxdy\\
\notag
+&\int_0^{\infty}\frac{\partial {\operatorname{Gap}_{2}}}{\partial \ell_1}\cdot\frac{1}{2^{\delta(g,m,n-1)}}\sum_{j=2}^{m}V_{g,m,n-1}(x,\ell_{1},\cdots,\ell_{j-1},\ell_{j+1},\cdots,\ell_{m},\theta_2,\cdots\theta_n)dx\\
\notag
+&\int_0^{\infty}\frac{\partial {\operatorname{Gap}_{3}}}{\partial \ell_1}\cdot\frac{1}{2^{\delta(g,m+1,n2)}}\sum_{i=1}^{n}V_{g,m+1,n-2}(x,\ell_{1},\cdots,\ell_{m},\theta_2,\cdots,\theta_{i-1},\\\notag&\theta_{i+1},\cdots,\theta_n)dx,
\end{align}

where :
\begin{equation}\notag
    \operatorname{Gap}_{3}(L;\alpha,\beta))=\frac{L}{2}-\tanh^{-1}\left(\frac{\sinh\frac{L}{2}\sinh{\frac{|\beta|}{2}}}{\cos\frac{\phi}{2}+\cosh{\frac{L}{2}}\cosh\frac{|\beta|}{2}}\right).
\end{equation}
And we have the following equation $(\ref{B3})$:
\begin{equation}\label{B3}
    \frac{\partial {Gap_{3}}}{\partial \theta}(\theta,x,y)=\frac{\partial {Gap_3}}{\partial \ell_1}(i\theta,x,y).
\end{equation}

 As $\delta$ depends on the number of the handles separated by $\gamma$, we have the following equation $(\ref{B4})$:
\begin{align}
    \delta(g,n-1,m+1)=\delta(g,n,m).
    \label{B4}
\end{align}
 
 By  induction, we have the following equation $(\ref{B5})$:
\begin{align}\notag
V_{g-1,m+2,n-1}(x,y,\ell_2,\cdots,\ell_{m+1},\theta_2,\cdots,\theta_n)\\=V_{g-1,m+n+1}(x,y,\ell_2,\cdots,\ell_{m+1},i\theta_2,\cdots,i\theta_n).\label{B5}
\end{align}
Using $(\ref{B3}),(\ref{B4}),(\ref{B5})$, we obtain the following:
\begin{align}\notag
    \frac{\partial({\frac{1}{2}{\ell_{m+1} V_{g,m+1,n-1}}})}{\partial{\ell_{m+1}}}(\ell_1,\cdots,\ell_{m+1}=i\theta_1,\theta_2,\cdots,\theta_{n})\\\notag
    =\frac{\partial({\frac{1}{2}{\theta_1 V_{g,m,n}}})}{\partial{\theta_1}}(\ell_1,\cdots,\ell_m,\theta_1,\theta_2,\cdots,\theta_{n}).
\end{align}

Then by induction, we know that the surfaces with $n$ cone points have the recursion formula.
\end{proof}
\section*{Funding}
This work was supported by National Natural Science Foundation of China [grant number 12271533, 12371076].
\section*{Acknowledgments} We thank Professor Zhe Sun for raising the question of volumes of moduli spaces of compact hyperbolic surfaces with cone points to us and many helps. We thank Professor Weixu Su for helpful suggestions and helps. 
\clearpage

\newpage

\bibliography{ref}

\begin{thebibliography}{10}

\bibitem{anagnostou2023weilpeterssonvolumesstabilityconditions}
Lukas Anagnostou, Scott Mullane, and Paul Norbury.
\newblock Weil-{P}etersson volumes, stability conditions and wall-crossing.
\newblock {\em arXiv:2310.13281}, 2023.

\bibitem{anagnostou2024volumesmodulispaceshyperbolic}
Lukas Anagnostou and Paul Norbury.
\newblock Volumes of moduli spaces of hyperbolic surfaces with cone points.
\newblock {\em arXiv:2212.13701}, 2024.

\bibitem{MR3497292}
Martin Bridgeman and Ser~Peow Tan.
\newblock Identities on hyperbolic manifolds.
\newblock In {\em Handbook of {T}eichm\"uller theory. {V}ol. {V}}, volume~26 of {\em IRMA Lect. Math. Theor. Phys.}, pages 19--53. Eur. Math. Soc., Z\"urich, 2016.

\bibitem{buser1992geometry}
P.~Buser.
\newblock {\em Geometry and Spectra of Compact Riemann Surfaces}.
\newblock Geometry and Spectra of Compact Riemann Surfaces. Birkh{\"a}user, 1992.

\bibitem{MR2520065}
Norman Do and Paul Norbury.
\newblock Weil-{P}etersson volumes and cone surfaces.
\newblock {\em Geom. Dedicata}, 141:93--107, 2009.

\bibitem{MR762512}
William~M. Goldman.
\newblock The symplectic nature of fundamental groups of surfaces.
\newblock {\em Adv. in Math.}, 54(2):200--225, 1984.

\bibitem{goncharov2024exponentialvolumesmodulispaces}
Alexander~B. Goncharov and Zhe Sun.
\newblock Exponential volumes of moduli spaces of hyperbolic surfaces, 2024.

\bibitem{MR4595283}
Yi~Huang and Zhe Sun.
\newblock Mc{S}hane identities for higher {T}eichm\"uller theory and the {G}oncharov-{S}hen potential.
\newblock {\em Mem. Amer. Math. Soc.}, 286(1422):v+116, 2023.

\bibitem{MR3389436}
Greg McShane.
\newblock {\em A remarkable identity for lengths of curves}.
\newblock PhD thesis, University of Warwick (United Kingdom), 1991.

\bibitem{MR1625712}
Greg McShane.
\newblock Simple geodesics and a series constant over {T}eichmuller space.
\newblock {\em Invent. Math.}, 132(3):607--632, 1998.

\bibitem{MR2264808}
Maryam Mirzakhani.
\newblock Simple geodesics and {W}eil-{P}etersson volumes of moduli spaces of bordered {R}iemann surfaces.
\newblock {\em Invent. Math.}, 167(1):179--222, 2007.

\bibitem{MR1844999}
Toshihiro Nakanishi and Marjatta N\"a\"at\"anen.
\newblock Areas of two-dimensional moduli spaces.
\newblock {\em Proc. Amer. Math. Soc.}, 129(11):3241--3252, 2001.

\bibitem{MR2215456}
Ser~Peow Tan, Yan~Loi Wong, and Ying Zhang.
\newblock Generalizations of {M}c{S}hane's identity to hyperbolic cone-surfaces.
\newblock {\em J. Differential Geom.}, 72(1):73--112, 2006.

\bibitem{Thurston1979TheGA}
William~P. Thurston.
\newblock {\em The geometry and topology of three-manifolds}.
\newblock Princeton lecture notes, 1979.

\bibitem{MR657237}
Scott Wolpert.
\newblock The {F}enchel-{N}ielsen deformation.
\newblock {\em Ann. of Math. (2)}, 115(3):501--528, 1982.

\bibitem{MR2290758}
Don Zagier.
\newblock The dilogarithm function.
\newblock In {\em Frontiers in number theory, physics, and geometry. {II}}, pages 3--65. Springer, Berlin, 2007.

\bibitem{Zhang2004}
Ying Zhang.
\newblock {\em Hyperbolic Cone-Surfaces, Generalized Markoff Maps, Schottky Groups and {M}cShane's Identity}.
\newblock PhD thesis, Department of Mathematics, National University of Singapore, Singapore, 2004.

\end{thebibliography}
\bibliographystyle{plain}

\end{document}